\documentclass[preprint,12pt]{elsarticle}
\usepackage{graphicx}
\usepackage[cp1251]{inputenc}
\usepackage{amssymb}
\usepackage{amsthm}
\usepackage{cmap}

\usepackage{amsmath}% http://ctan.org/pkg/amsmath

\biboptions{sort&compress}

\newcommand{\sysn}{\left\{\begin{array}{rcl}}
\newcommand{\sysk}{\end{array}\right.}

\newtheorem{theorem}{Theorem}[section]

\theoremstyle{example}

\newtheorem{proposition}[theorem]{Proposition}
\theoremstyle{definition}
\newtheorem{definition}[theorem]{Definition}
\newtheorem{corollary}[theorem]{Corollary}

\journal{...}

\begin{document}

\title{The $\Delta_1$-property of $X$ is equivalent to the Choquet property of $B_1(X)$}

\author{Alexander V. Osipov}

\address{Krasovskii Institute of Mathematics and Mechanics, \\ Ural Federal
 University, Yekaterinburg, Russia}

\ead{OAB@list.ru}

\begin{abstract}

 We give a characterization of the $\Delta_1$-property of any
Tychonoff space $X$ in terms of the function space $B_1(X)$ of all Baire-one real-valued
functions on a space  $X$ with the topology of pointwise convergence. We establish that for a Tychonoff space $X$ the $\Delta_1$-property  is equivalent to the Choquet property of $B_1(X)$. Also we construct under $ZFC$ an example
of a separable pseudocompact space $X$ such that $C_p(X)$
is $\kappa$-Fr\'{e}chet-Urysohn but $X$ fails to be a $\Delta_1$-space.  This answers a question of K\c{a}kol-Leiderman-Tkachuk.
\end{abstract}
%\tnotetext[label1]{The research has been supported by .}


\maketitle %%
% \linenumbers


\section{Introduction}

A family $\{A_{\alpha}: \alpha\in \kappa\}$ of subsets of a set $X$ is said to be {\it point-finite} if for every $x\in X$, $\{\alpha\in \kappa: x\in A_{\alpha}\}$ is finite.

A space $X$ has the {\it $\Delta_1$-property} \cite{KKL} if any disjoint sequence $\{A_n: n\in \omega\}$ of
countable subsets of $X$ has a point-finite open expansion, i.e., there exists a point-finite family $\{U_n: n\in \omega\}$ of open subsets of $X$ such that
$A_n\subset U_n$ for every $n\in \omega.$  The spaces with the $\Delta_1$-property are also called $\Delta_1$-spaces.

The $\Delta_1$-property was intoduced and studied systematically in \cite{KKL} and \cite{KLT}.

A space $X$ is {\it Fr\'{e}chet-Urysohn} if, for any $A\subseteq X$ and any $x\in A$, there exists
a sequence $\{a_n: n\in \omega\}\subseteq A$ that converges to $x$.

A space $X$ is called {\it $k$-Fr\'{e}chet-Urysohn} if, for any open set $U\subset X$ and any point $x\in \overline{U}$, there exists a sequence $\{x_n: n\in \omega\}\subset U$ that converges to~$x$. Clearly, every Fr\'{e}chet-Urysohn space is $k$-Fr\'{e}chet-Urysohn.

\medskip
A family $\{A_{\alpha}: \alpha\in \kappa\}$ of subsets of a space $X$ is said to be {\it strongly point-finite} if for every $\alpha\in \kappa$, there exists an open
set $U_{\alpha}$ of $X$ such that $A_{\alpha}\subset U_{\alpha}$ and $\{U_{\alpha}: \alpha\in \kappa\}$ is point-finite.

A space $X$ is said to have {\it property $(\kappa)$} if every pairwise
disjoint sequence of finite subsets of $X$ has a strongly point-finite subsequence.

\medskip

 In \cite{Sakai}, Sakai characterized $\kappa$-Fr\'{e}chet-Urysohn property in $C_p(X)$ showing that $C_p(X)$ is $\kappa$-Fr\'{e}chet-Urysohn if and only if $X$ has the property $(\kappa)$.

\medskip

It was also shown in \cite{KKL} that $C_p(X)$ is $\kappa$-Fr\'{e}chet-Urysohn whenever $X$ has the $\Delta_1$-property and,
under Martin's Axiom, even the Fr\'{e}chet-Urysohn property of $C_p(X)$ does not
imply that $X$ is a $\Delta_1$-space.

\medskip

 A space is {\it meager} (or {\it of the first Baire category}) if it
can be written as a countable union of closed sets with empty
interior. A topological space $X$ is {\it Baire} if the Baire
Category Theorem holds for $X$, i.e., the intersection of any
sequence of open dense subsets of $X$ is dense in $X$. Clearly, if
$X$ is a Baire space, then $X$ is not meager.

\medskip

In \cite{Os3} it is proved that $X$ has the property $(\kappa)$  if and only if $B_1(X)$ is Baire.

\medskip

Both Baire and meager space have game characterizations due to
Oxtoby \protect\cite{Ox}.

\medskip
The game $G_I(X)$ is started by the player ONE who selects a
nonempty open set $V_0\subseteq X$. Then the player TWO responds
selecting a nonempty open set $V_1\subseteq V_0$. At the $n$-th
inning the player ONE selects a nonempty open set $V_{2n}\subseteq
V_{2n-1}$ and the player TWO responds selecting a nonempty open
set $V_{2n+1}\subseteq V_{2n}$. At the end of the game, the player
ONE is declared the winner if $\bigcap\limits_{n\in \omega} V_n$
is empty. In the opposite case the player TWO wins the game
$G_I(X)$.

The game $G_{II}(X)$ differ from the game $G_I(X)$ by the order of
the players. The game $G_{II}(X)$ is started by the player TWO who
selects a nonempty open set $V_0\subseteq X$. Then player ONE
responds selecting a nonempty open set $V_1\subseteq V_0$. At the
$n$-th inning the player TWO selects a nonempty open set
$V_{2n}\subseteq V_{2n-1}$ and the player ONE responds selecting a
nonempty open set $V_{2n+1}\subseteq V_{2n}$. At the end of the
game, the player ONE is declared the winner if
$\bigcap\limits_{n\in\omega} V_n$ is empty. In the opposite case
the player TWO wins the game $G_{II}(X)$.

The following classical characterizations can be found in
\protect\cite{Ox}.

\medskip

A topological space $X$ is

(1) meager if and only if the player ONE has a winning strategy in
the game $G_{II}(X)$;

(2) Baire if and only if the player ONE has no winning strategy in
the game $G_{I}(X)$.

\medskip

A topological space $X$ is defined to be {\it Choquet} if the
player TWO has a winning strategy in the game $G_{I}(X)$. Choquet
spaces were introduced in 1975 by White \cite{Whi} who called them
{\it weakly $\alpha$-favorable spaces}.

\medskip

We establish that the $\Delta_1$-property of $X$ is equivalent to the Choquet property of $B_1(X)$ and  we construct under $ZFC$ an example
of a separable pseudocompact space $X$ such that $C_p(X)$
is $\kappa$-Fr\'{e}chet-Urysohn but $X$ fails to be a $\Delta_1$-space.  This answers in the positive
Question 4.7 in \cite{KLT}.

\section{Main definitions and notation}

In this paper all spaces are assumed to be Tychonoff. The expression $C(X,Y)$ denotes the set of all continuous maps from a space $X$
to a space $Y$. We follow the usual practice to write $C(X)$ instead of $C(X,\mathbb{R})$. The
space $C_p(X)$ is the set $C(X)$ endowed with the pointwise convergence topology.

A real-valued function $f$ on a space $X$ is a {\it Baire-one
function} (or a {\it function of the first Baire class}) if $f$ is
a pointwise limit of a sequence of continuous functions on $X$.
Let $B_1(X)$ denote the space of all Baire-one real-valued
functions on a space  $X$ with the topology of pointwise convergence.

 We recall that a subset of $X$ that is the
 complete preimage of zero for a certain function from~$C(X)$ is called a zero-set.
A subset $O\subseteq X$  is called  a cozero-set (or functionally
open) of $X$ if $X\setminus O$ is a zero-set of $X$. It is easy to
check that zero sets are preserved by finite unions and countable
intersections. Hence cozero sets are preserved by finite
intersections and countable unions. Countable unions of zero sets
will be denoted by $Zer_{\sigma}$ (or $Zer_{\sigma}(X)$),
countable intersection of cozero sets by $Coz_{\delta}$ (or
$Coz_{\delta}(X)$). It is easy to check that $Zer_{\sigma}$-sets
are preserved by countable unions and finite intersections. It is well known that $f$ is
of the first Baire class if and only if $f^{-1}(U)\in
Zer_{\sigma}$ for every open $U\subseteq \mathbb{R}$ (see Exercise
3.A.1 in \protect\cite{lmz1}).

\medskip

A $Coz_{\delta}$-subset of $X$ containing $x$ is called a {\it
$Coz_{\delta}$ neighborhood} of $x$.

\medskip

 A set $A\subseteq X$ is called {\it strongly
$Coz_{\delta}$-disjoint}, if there is a pairwise disjoint
collection $\{F_a: F_a$ is a $Coz_{\delta}$ neighborhood of $a$,
$a\in A\}$  such that $\{F_a: a\in A\}$ is a {\it completely
$Coz_{\delta}$-additive system}, i.e. $\bigcup\limits_{b\in B}
F_b\in Coz_{\delta}$ for each $B\subseteq A$.

A disjoint sequence $\{\Delta_n: n\in \omega\}$ of (finite)
sets is  {\it strongly $Coz_{\delta}$-disjoint} if the set
$\bigcup\{\Delta_n: n\in \omega\}$ is strongly
$Coz_{\delta}$-disjoint (see Definition 1 in \cite{Osip1}).

\medskip

 In (\cite{Osip1}, see Theorem 4.3),  it is proved that $B_1(X)$ is Choquet (pseudocomplete) if and only if every countable subset of $X$ is strongly
$Coz_{\delta}$-disjoint.

\medskip

The sequence $\{\mathcal{C}_n : n\in \omega\}$ is called {\it
pseudocomplete} if, for any family $\{U_n: n\in \omega\}$ such
that $\overline{U_{n+1}}\subseteq U_n$ and we have $U_n\in
\mathcal{C}_n$ for each $n\in\omega$, we have $\bigcap \{U_n:
n\in\omega\}\neq \emptyset$. A space $X$ is called {\it
pseudocomplete} if there is a pseudocomplete sequence
$\{\mathcal{B}_n : n\in\omega\}$ of $\pi$-bases in $X$.

\medskip

It is a well-known that any pseudocomplete space is Baire and any
$\check{C}$ech-complete space is pseudocomplete. Note that if $X$
has a dense pseudocomplete subspace (in particular, if $X$ has a
dense $\check{C}$ech-complete subspace) then $X$ is pseudocomplete
(p. 47 in \protect\cite{Tk}).

\medskip

\section{Main results}

\begin{theorem}\label{th1} For a Tychonoff space $X$, the following conditions are equivalent:

\begin{enumerate}

\item  $X$ has the $\Delta_1$-property.

\item   $B_1(X)$ is Choquet.

\item  $B_1(X)$ is pseudocomplete.

\item Every countable subset of $X$ is strongly
$Coz_{\delta}$-disjoint.

\end{enumerate}

\end{theorem}

\begin{proof} $(1)\Rightarrow(2)$. Let $S$ be a countable subset of $X$. Let us represent the set $S=\{s_i: i\in \omega\}$. Then the sequence $(s_i)_{i\in \omega}$
has a point-finite open expansion, i.e.  for every $i\in \omega$, there exists an open
set $U_{i}$ of $X$ such that $s_{i}\in U_{i}$ and $\{U_{i}: i\in \omega\}$ is point-finite.
Since $X$ is functionally Hausdorff (Tychonoff), for any pair $x,y\in S$ ($x\neq y$) there is a continuous function $f_{xy}:X\rightarrow [0,1]$ such that $f_{xy}(x)=0$ and $f_{xy}(y)=1$. 

Since $X$ is Tychonoff, for every $i\in \omega$ there is a continuous function $f_{i}:X\rightarrow [0,1]$ such that  $f_{i}(s_{i})\subseteq\{0\}$ and $f_{i}(X\setminus U_{i})\subseteq\{1\}$.

For every  $s_i\in S$ denote by

$S_i=(\bigcap\{f^{-1}_{ys_i}(1): y\in  S\setminus \{s_i\}\})\cap (\bigcap\{f^{-1}_{s_iy}(0): y\in  S\setminus \{s_i\}\})\cap f^{-1}_i(0)$.

Since $f^{-1}_{s_js_i}(1)\cap f^{-1}_{s_js_i}(0)=\emptyset$, $S_i\cap S_j=\emptyset$ for any $i,j\in \omega$ ($i\neq j$).

Since $S$ is countable,  $S_i$ is a zero-set of $X$ and $s_{i}\in S_i\subseteq U_i$ for every $i\in\omega$.

For every $i\in\omega$ we consider the family $\{W_{i,j}: j\in \omega\}$ of co-zero sets of $X$ such that $S_i=\bigcap\{W_{i,j}: j\in \omega\}$,
$W_{i,j+1}\subset W_{i,j}$ and $W_{i,1}\subset U_i$ for any $j,i\in \omega$.

Consider the set $A=\bigcap\limits_{j\in \omega} \bigcup\limits_{i\in \omega} W_{i,j}$. Note that $\bigcup\limits_{i\in \omega} S_i\subseteq A$.

\bigskip

We claim that $A\subseteq \bigcup\limits_{i\in \omega} S_i$. Let $x\not\in \bigcup\limits_{i\in \omega} S_i$. Then the set $\{i: x\in U_i\}$ is finite. Let $\{i_1,...,i_m\}=\{i: x\in U_i\}$. For every $s\in \{1,...,m\}$ there is $j_s\in \omega$ such that $x\not\in W_{i_s,j_s}$. Denote by $l=\max\{j_s: s\in \{1,...,m\}\}$.
Then $x\not\in W_{i,l}$ for every $i\in \omega$ and, hence, $x\not\in A$ and  $A\subseteq \bigcup\limits_{i\in \omega} S_i$.

Thus $A=\bigcup\limits_{i\in \omega} S_i$ is a $Coz_{\delta}$-set of $X$.  Let $C\subset \omega$. Since $\bigcup\limits_{i\in \omega\setminus C} S_i$ is a $Zer_{\sigma}$-set, $\bigcup\limits_{i\in C} S_i=A\setminus (\bigcup\limits_{i\in \omega\setminus C} S_i)$ is a $Coz_{\delta}$-set of $X$.

It follows that the family $\{S_i: i\in \omega\}$ of zero-sets of $X$ is disjoint and completely $\mathrm{Coz}_{\delta}$-additive. Thus, $S$ is strongly
$Coz_{\delta}$-disjoint.

$(2)\Rightarrow(1)$.  Let $\{S_i: i\in \omega\}$ be a pairwise disjoint sequence of non-empty countable subsets of $X$.
Then $S=\bigcup S_i$ is strongly
$Coz_{\delta}$-disjoint, i.e. there is a pairwise disjoint collection $\{O_a : O_a$ is a $Coz_{\delta}$ neighborhood of $a$, $a\in S\}$ such that $\{O_a : a \in S\}$ is a completely $Coz_{\delta}$-additive system.

By Lemma 3.3 in \cite{Osip1}, we can assume that $O_a$ is a zero-set of $X$ for every $a\in S$.

Let $O=\bigcup \{O_{a}: a\in S\}$. Since $\{O_a : a \in S\}$ is completely $\mathrm{Coz}_{\delta}$-additive, $O$ is $\mathrm{Coz}_{\delta}$-set in $X$, i.e. there is a family $\{W_i: i\in \omega\}$ of co-zero sets of $X$ such that $O=\bigcap \{W_i: i\in \omega\}$ and $W_{i+1}\subset W_i$ for every $i\in \omega$.

Let $F_i=\bigcup\{O_{a}: a\in S_i\}$. Note that $F_i$ is a $Zet_{\sigma}$-set of $X$ for any $i\in \omega$.

Since $\{O_a : a \in S\}$ is completely $\mathrm{Coz}_{\delta}$-additive, the set $T_j=O\setminus \bigcup\limits_{i=1}^j F_i$ is a $\mathrm{Coz}_{\delta}$-set in $X$. Let $T_j=\bigcap \{W_{i,j}: i\in \omega\}$ where $\{W_{i,j}: i\in \omega\}$ is a family of co-zero sets of $X$ for every $j\in \omega$ such that such that $W_{i+1,j}\subset W_{i,j}$ for every $i\in \omega$.

Let $P_1=W_1$, $P_2=W_2\cap W_{1,1}$, $P_3=W_3\cap W_{2,1}\cap W_{1,2}$ and so on, i.e. $P_n=W_n\bigcap \{W_{n-i,i}: i\in \overline{1,n-1}\}$.

Note that $\{P_n: n\in \omega\}$ is a family co-zero sets of $X$ such that $S_n\subset P_n$ for every $n\in \omega$, $P_{n+1}\subset P_n$ and $\bigcap P_n=\emptyset$. Thus the family $\{S_n: n\in \omega\}$ has a point-finite open expansion $\{P_n: n\in \omega\}$.

$(2)\Leftrightarrow(3)\Leftrightarrow(4)$. By Theorems 4.4 and 4.5 in \cite{Osip1}.

\end{proof}

A topological space $X$ is defined to be
{\it $k$-scattered} if every compact Hausdorff subspace of $X$ is scattered.

A topological space $X$ is called {\it almost $K$-analytic} if every countable subset of $X$ is contained in a $K$-analytic $G_{\delta}$-subspace of $X$.

By Theorem 17 in \cite{Osip1}, we get the following.

\begin{corollary}{\it For an almost $K$-analytic space $X$, the following conditions are equivalent:

\begin{enumerate}

\item $C_p(X)$ is $\kappa$-Fr\'{e}chet-Urysohn.

\item $B_1(X)$ is Choquet.

\item $B_1(X)$ is Baire.

\item  $X$ has the property $(\kappa)$.

\item  $X$ has the $\Delta_1$-property.

\item $X$ is $k$-scattered.

\end{enumerate}}

\end{corollary}

\section{Solution of open questions}

\begin{definition}(Definition 3.3 in \cite{Osip2}). Let $\mathcal{S}=\{S_n : n\in \omega\}$ be a family of subsets of a space $X$ and $x\in X$. Then $\mathcal{S}$ {\it weakly converges} to $x$ if for every neighborhood $W$ of $x$ there is a sequence $(s_n : n\in \omega)$ such that $s_n\in S_n$ for each $n\in \omega$ and there is $n'$ such that $s_n\in W$ for each $n>n'$.
\end{definition}

\medskip

In \cite{KLT} an equivalent formulation is proposed: a family $\mathcal{S}=\{S_n : n\in \omega\}$ of subsets of a space $X$ {\it weakly
converges} to $x$ if for any neighborhood $W$ of $x$ there exists a finite set $S\subset \omega$
such that $S_n\cap W\neq \emptyset$  for any $n\in \omega\setminus S$.

\medskip

In this paper a family $\mathcal{S}=\{S_n : n\in \omega\}$ will be called a {\it weakly converging sequence} if $\mathcal{S}$ is countably infinite, disjoint, every element $S_n$ of $\mathcal{S}$ is a non-empty finite set and $\mathcal{S}$ weakly converges to some point $x\in X$.

\medskip

In (\cite{KLT}, Theorem 3.19), it is proved that if $X$ is a maximal space, then it has no weakly convergent sequences.

\medskip

 A non-empty space is called {\it crowded} if it
has no isolated points.

\medskip

 %Anton Lipin noted in a personal conversation that the following result is correct.

\medskip

\begin{proposition}\label{pr1} If $X$ is a crowded submaximal space, then it has no weakly convergent sequences.
\end{proposition}

\begin{proof}  Let $\mathcal{S}=\{S_n : n\in \omega\}$ be a disjoint family of  finite subsets of a space $X$. Let $\mathcal{H}$ be a maximal almost disjoint family of infinite subsets of $\omega$. For every $E\in \mathcal{H}$ we consider $A_E=\bigcup\limits_{n\in E} S_n$. Since $A_E\cap A_F$ is finite for any $E,F\in \mathcal{H}$ ($E\neq F$), $Int A_E\cap Int A_F=\emptyset$. Since $\bigcup\limits_{E\in \mathcal{H}} A_E$ is countable, there is $E\in \mathcal{H}$ such that $Int A_E=\emptyset$. Since $X$ is submaximal, the set $A_E$ has no limit points. It follows that $\mathcal{S}$ cannot weakly converge to any point of $X$.
\end{proof}

This answers in the positive
Question 4.9 of \cite{KLT}. By Corollary 3.18 in \cite{KLT} and Theorem 3.1 in \cite{Os3}, we get the following corollaries.

\begin{corollary} If $X$ is a crowded submaximal space, then $X$ has the property $(\kappa)$.
\end{corollary}

This answers in the positive
Question 4.10 of \cite{KLT}.

\begin{corollary} If $X$ is a crowded submaximal space, then $B_1(X)$ is Baire.
\end{corollary}

In (\cite{JSS}, Theorem 4.1) it is proved under $ZFC$ that for each infinite cardinal $\kappa$ the Cantor cube $D^{2^{\kappa}}$
contains a dense submaximal subspace $X$ with $|X|=\Delta(X)=\kappa.$

\medskip
Thus in the case $k=\omega$ there is a countable dense subset in the Cantor cube $D^{\mathfrak{c}}$ which has no weakly convergent sequences. This answers in the positive
Question 4.8 of \cite{KLT}. It follows that by Theorem 3.21 in \cite{KLT} the following result holds.

\begin{theorem} There exists in $ZFC$
a separable pseudocompact dense subspace $P\subset D^{\mathfrak{c}}$ which is not a $\Delta_1$-space but
has the property $(\kappa)$.
\end{theorem}

\medskip

This answers in the positive
Question 4.7 of \cite{KLT}. By Theorem \ref{th1} and Theorem 3.1 in \cite{Os3} we get the following

\begin{corollary} There exists in $ZFC$
a separable pseudocompact dense subspace $P\subset D^{\mathfrak{c}}$ such that $B_1(P)$ is Baire, but $B_1(P)$ is not Choquet.
\end{corollary}

\medskip

Let $\tau_{uc}$  be the topology of uniform convergence on compacta.

\medskip

{\bf Question.}  Assume that $X$ a Tychonoff space. Is is true that $X$ has the $\Delta_2$-property  if and only if $(B_1(X), \tau_{uc})$ is Choquet?

\medskip

Definition of $\Delta_2$-property see (\cite{KKL}, Def 1.3).

\medskip

{\bf Acknowledgements.} The author would like to thank Anton Lipin for discussions of Proposition \ref{pr1} and Evgenii Reznichenko for the motivation for this paper.

The author would like to thank the
referee for valuable comments.

%\end{fulltext}

\bibliographystyle{model1a-num-names}
\bibliography{<your-bib-database>}

\end{document}